\documentclass[12pt]{article}

\usepackage{amssymb,graphicx,amsthm}
\usepackage{amsmath}
\usepackage{makecell}

\newtheorem{thm}{Theorem}[section]
\newtheorem{lemma}[thm]{Lemma}

\begin{document}

\title{{\bf A genealogy of the translation planes of order 25}}
\author{Jeremy M. Dover\\
Dover Networks LLC\\
\tt doverjm@gmail.com}

\maketitle

\begin{abstract}
In 1992 Czerwinski and Oakden (The translation planes of order 25, {\em J. Combin. Theory Ser. A}, 59:193-217, 1992) provided an exhaustive list of all spreads of $PG(3,5)$ and thus of all translation planes of that order. At that time, the authors provided a partial correlation of these planes to those then-described in the literature, but the intervening years have provided additional construction techniques and classification results. This paper provides an extensive classification of these planes against the currently-known construction techniques, finding two planes that do not belong to any current infinite family. The author provides additional details for these two planes to help put them in context, as a spur for further research.
\end{abstract}

\section{Introduction}
The seminal paper of Bruck and Bose~\cite{bruckbose} provides a correspondence between translation planes and spreads of projective spaces of odd dimension, and using this correspondence in 1992 Czerwinski and Oakden~\cite{czoak} provided a complete list of spreads of $PG(3,5)$, and thus a complete list of translation planes of order 25. Their list comprised five subregular spreads, which include the regular spread and all Andr\'e spreads; eight spreads which are not subregular but contain one or more reguli, and are thus derivable in the sense of Ostrom~\cite{ostrom}; and eight spreads which contain no regulus.  The translation planes associated with each of these three classes are denoted $S_1 \ldots S_5$, $A_1 \ldots A_8$, and $B_1 \ldots B_8$, respectively. Czerwinski and Oakden identified fourteen of these planes as having appeared elsewhere in the literature, though their analysis overlooked planes generated by several then-existing constructions, particularly spreads obtained from flocks of quadratic cones and nest replacement.

Moorhouse~\cite{moorhouse} has explicitly constructed all of the translation planes of order 25, providing models for each plane at the website~\cite{moorhouse:web}. Moorhouse has computed a variety of invariants for all of these planes, including the 5-rank of their incidence matrices, and the order and orbit structure of their automorphism groups. Moreover, Moorhouse has determined all of the different ways of transforming these planes, using dualization and derivation, potentially applied multiple times, into other non-translation planes and between themselves.

With so much known about the planes, another look at them seems unnecessary. However, the existing presentations of these planes tend to treat the planes as individual objects, but in many cases these planes exist in the context of infinite families of planes generated by algebraic, geometric or combinatorial methods. In this paper, we want to identify how these planes fit in with the many infinite classes of projective planes that have been constructed in the literature over the years. In some sense, we are attempting to address the question of ``why" these planes exist.

\section{Constructing Planes of Order 25}

Rather than starting with the known translation planes of order 25 and attempting to categorize them, our approach starts with identifying plane construction techniques and applying them to create planes of order 25; Johnson, et al.'s~\cite{jjb} {\em Handbook of Finite Translation Planes} has been an invaluable resource. We then classify the resulting planes based on calculated invariants, extensively leveraging the data reported by Moorhouse~\cite{moorhouse:web}; we have also used Magma~\cite{magma} to create each of the 21 translation planes of order 25 and verify Moorhouse's calculations. One advantage of this approach is that in several cases we note collapsing between constructions, where seemingly disparate construction techniques yield the same plane. However, our emphasis on classifying these planes into infinite families elides the fact that many of these planes were initially discovered through other techniques; please do not interpret the results of this section as any sort of assertion of primary discovery.

\subsection{Subregular Planes}
As defined by Bruck~\cite{bruck}, subregular planes are those that can be obtained from a spread of $PG(3,q)$ by a series of regulus reversals starting with a regular spread. Orr~\cite{orr} showed that every subregular plane can be obtained by starting with a regular spread and reversing a set of pairwise disjoint reguli. Thus the problem of finding all subregular planes is equivalent to finding all projectively inequivalent sets of pairwise disjoint reguli in a regular spread, and as reported by Czerwinski and Oakden, we find that $S_1$ is the Desarguesian plane, $S_2$ is the Hall plane~\cite{hall}, $S_3$ and $S_4$ are Andr\'e planes~\cite{andre}, and $S_5$ is a subregular plane that is not an Andr\'e plane~\cite{bruck}.

\subsection{Nearfield Planes}
There are two nearfields of order 25: the Dickson nearfield and the exceptional nearfield labelled ``I" by Zassenhaus~\cite{zassenhaus}. Czerwinski and Oakden report that the regular nearfield plane is $A_2$ and the exceptional nearfield plane is $S_4$, but our calculations suggest the opposite. Indeed, it is easy to see from the definition of the regular nearfield of order 25 that it is an Andr\'e quasifield, and thus the corresponding plane must be subregular.

\subsection{Flag-Transitive Planes}
Foulser~\cite{foulser} provided a construction of two flag-transitive translation planes of order 25, which Czerwinski and Oakden report as $B_1$ and $B_2$. Foulser's planes were generalized to an infinite family by Baker and Ebert~\cite{be:flagtrans}, and using their construction method we confirmed that these are the only two flag-transitive planes of order 25.

\subsection{Hering Quasifield Planes}
Hering~\cite{hering} provides an explicit description of an infinite family of quasifields, of which there is one example of order 25. By construction we confirmed Czerwinski and Oakden's assertion that it is $B_4$.

\subsection{Rao-Satyanarayana Planes of characteristic 5}
Rao and Satyanarayana~\cite{rao17} give a construction of a spread of $PG(3,5^r)$ for odd $r$ via spread sets of matrices, and there is one example of order 25. This is $A_7$, which is reported by Czerwinski and Oakden and we have confirmed.

\subsection{Rao, Rodabaugh, Wilke and Zemmer Planes}
Rao, Rodabaugh, Wilke and Zemmer~\cite{rao16} provide a construction of translation planes by using ideas similar to Foulser's $\lambda$-planes~\cite{foulserlambda} with the exceptional nearfields. For order 25, there are two such planes. One is $A_4$, which is asserted by Czerwinski and Oakden and we confirmed. The other is $B_5$, which was also found by Walker~\cite{walkersporadic}.

\subsection{Fisher flocks, Walker planes and Baker-Ebert $q$-nests}
The Fisher flocks~\cite{fisher} are flocks of a quadratic cone in $PG(3,q)$ for odd $q$ which can be used to construct spreads via the Klein correspondence. Walker~\cite{walkerfamily} provides a different construction which is equivalent to a flock of a quadratic cone, and which coincides with the Fisher flock when $q=5$. Baker and Ebert~\cite{be:qnest} provide an alternative, but equivalent, formulation of the same spread by replacing a $q$-nest in a regular spread. The corresponding translation plane is $A_3$. Since the spread arises from a flock of a quadratic cone, it is the union of $q$ reguli which pairwise meet in a single distinguished line of the spread. One can perform derivation on any of these reguli, but all of the resulting translation planes are isomorphic to $A_7$. Moorhouse~\cite{moorhouse} confirms that there are no other ways to derive this plane.

\subsection{Pabst-Sherk planes and Ebert $(q+1)$-nests}
Pabst and Sherk~\cite{pabstsherk} use the method of indicator sets to construct a family of translation planes, for which Ebert~\cite{ebert:q+1nest} provides an equivalent formulation by replacing a $(q+1)$-nest of a regular spread. There are two non-isomorphic planes obtained from this construction, and these are $A_2$ and $A_4$. Following Baker, et al.~\cite{bdew}, both of these spreads are spawned from a non-linear hyperbolic fibration, i.e. a set of $q-1$ pairwise disjoint hyperbolic quadrics which with two additional disjoint lines partition $PG(3,5)$. There is only one non-linear hyperbolic fibration of $PG(3,5)$ up to isomorphism: De Clerck et al.~\cite{dgt} prove that there is only one non-linear flock of the quadratic cone in $PG(3,5)$, and Baker et al.~\cite{bep} show that there is a bijection between flocks of the quadratic cone and hyperbolic fibrations. Therefore, these two planes must both come from spreads spawned from the same hyperbolic fibration. This hyperbolic fibration can spawn 16 different spreads by independently picking one of the two ruling families of each hyperbolic quadric in the fibration, but in this case there is a great deal of isomorphism between these spreads, and the only additional plane obtained from this hyperbolic fibration is $A_8$.

\subsection{Bruen chains and Baker-Ebert $(q-1)$-nests}
Bruen chains~\cite{bruen} appear to occur sporadically for different values of $q$, but in $PG(3,5)$ the only Bruen chain is exactly a Baker-Ebert $(q-1)$-nest~\cite{be:q-1nest}. The plane obtained by replacing such a nest is $A_1$. Using Baker and Ebert's formulation, we see that we can obtain additional planes from this plane in several ways. First, there is a companion $(q-1)$-nest which is disjoint from the original, and thus can be paired with the original to create a $2(q-1)$-nest. Replacing this $2(q-1)$-nest with either of its two possible nest replacements yields the plane $B_6$. On the other hand, there are two types of reguli disjoint from the original $(q-1)$-nest, namely the reguli in the companion nest, and some reguli in the bundle that defines the nest. Reversing a regulus of the companion $(q-1)$-nest lets us derive the $(q-1)$-nest plane to get $A_5$. However, we can also derive one of the other reguli disjoint from the nest to obtain $A_6$. No other plane can be thus derived, consistent with Moorhouse~\cite{moorhouse}.

\subsection{Baker-Ebert mixed nests}
Baker and Ebert~\cite{be:nestgap} provide a construction of replaceable nests of a range of sizes, but in $PG(3,5)$ the only such nest not previously described is a 5-nest, the reversal of which yields $A_6$. These mixed nest spreads are known to be derivable, as they contain at least $\frac12 (q-3)$ pairwise disjoint reguli. Moorhouse~\cite{moorhouse} reports that $A_6$ can be derived from $A_1$, which we confirm.

\subsection{Baker-Ebert-Weida tabs}
Baker et al.~\cite{bew} provide several constructions of nests by taking various unions of Bruen chains. For example, two chains which meet in a single regulus are called a {\em single tab}, and the authors show that a single tab is a replaceable $(q+1)$-nest. In this paper, the authors construct a single tab in $PG(3,5)$ which can be reversed to create the plane $A_2$. However, we have determined computationally that there are two ways to create a single tab from Bruen chains, and the other creates the plane $B_7$.

Table~\ref{T1} summarizes the situation, ordered according to Czerwinski and Oakden's taxonomy.
\noindent
\begin{table}
\centering
\caption{The translation planes of order 25}{\label{T1}}
\begin{tabular}{|c|c|c|l|}
\hline
Label & p-rank & Group Order & Constructions\\ \hline\hline
$S_1$ & 226 & 304668000000 & Desarguesian\\ \hline
$S_2$ & 251 & 3600000 & Hall\\ \hline
$S_3$ & 260 & 720000 & Andr\'e\\ \hline
$S_4$ & 258 & 1440000 & \makecell{Andr\'e \\ Regular nearfield}\\ \hline
$S_5$ & 259 & 720000 & Subregular\\ \hline\hline
$A_1$ & 262 & 360000 & \makecell{Bruen chain \\ $(q-1)$-nest}\\ \hline
$A_2$ & 255 & 2880000 & \makecell{Exceptional nearfield \\ Pabst-Sherk/$(q+1)$-nest}\\ \hline
$A_3$ & 253 & 1500000 & Fisher flock\\ \hline
$A_4$ & 256 & 180000 & \makecell{Rao, Rodabaugh, Wilke, Zemmer\\Pabst-Sherk/$(q+1)$-nest}\\ \hline
$A_5$ & 259 & 60000 & Derived $(q-1)$-nest\\ \hline
$A_6$ & 259 & 360000 & Mixed nest\\ \hline
$A_7$ & 260 & 300000 & Derived Fisher flock\\ \hline
$A_8$ & 257 & 120000 & $(q+1)$-nest Hyperbolic Fibration\\ \hline\hline
$B_1$ & 258 & 130000 & Flag-transitive\\ \hline
$B_2$ & 262 & 130000 & Flag-transitive\\ \hline
$B_3$ & 264 & 90000 & \\ \hline
$B_4$ & 239 & 1800000 & Hering\\ \hline
$B_5$ & 261 & 4800000 & Rao, Rodabaugh, Wilke, Zemmer\\ \hline
$B_6$ & 258 & 120000 & $2(q-1)$-nest\\ \hline
$B_7$ & 261 & 240000 &  Baker-Ebert-Weida single tab\\  \hline
$B_8$ & 262 & 80000 & \\ \hline
\end{tabular}
\end{table}

\section{A new construction for replaceable nests}
By far the most interesting feature of Table~\ref{T1} is the fact that none of the constructions we attempted generated planes $B_3$ or $B_8$. As stated earlier, by starting with the constructions and utilizing all possible parameters for those constructions, we know that we did not miss any planes of types we have already examined. We know that there are no unexamined flocks of quadratic cones or hyperbolic fibrations in $PG(3,5)$, but while we have exhausted the known families of nests, there may be additional nests not in currently described families. With this in mind, we developed an exhaustive search to find nests of reguli in the regular spread of $PG(3,5)$. The search was a straightforward depth-first search with early isomorph rejection, and our search quickly identified 14 distinct nests, up to isomorphism. Of these nests, three are not Bruck-replaceable (see below for definition), and six are already described in Table~\ref{T1}. The remaining five nests can be replaced to obtain planes $S_5$, $B_3$, $B_4$, $B_5$ and $B_7$. Notice in particular that these nests fill one of the remaining two holes in the table. In what follows, we develop a technique to create new replaceable nests whose replacement generates $B_3$. This technique is foreshadowed in Baker, et al.~\cite{bew}, which proves the analogous results for Bruen chains.

We begin with some terminology, mostly following Baker and Ebert~\cite{be:qnest}. A {\em $t$-nest} of reguli in a regular spread $S$ of $PG(3,q)$ is a set of $t$ reguli such that every line contained in at least one regulus in the nest is contained in exactly two reguli of the nest. For simplicity of notation, we will often identify a nest as both a set of reguli and as the set of lines contained in those reguli. Nests were originally conceived as a generalization of Bruen chains, and the intended replacement methodology was to use opposite half-reguli to conduct the replacement. However, Prince~\cite{prince} has recently created nests which can be replaced, but not strictly with half-reguli. To accomodate these examples, let $N$ be a $t$-nest of the regular spread $S$ of $PG(3,q)$, and let $R_1 \ldots R_t$ be the reguli in $N$, with $R'_1 \ldots R'_t$ the corresponding opposite reguli. We say $N$ is {\em hemi-replaceable} if there exist $t$ sets $S_1 \ldots S_t$ such that for all $i$ $S_i \subset R'_i$ and $|S_i| = \frac12 (q+1)$, and the lines in $\bigcup_{i=1}^t S_i$ are pairwise disjoint. The set of $S_i$'s is called a {\em hemi-replacement} for $N$. Note that Weida~\cite{weida} has shown that any replaceable $t$-nest with $t \leq q$ is hemi-replaceable.

One natural way to create sets $S_i$ meeting the criteria for hemi-replaceability is to use the Bruck kernel of the regular spread $S$. The Bruck kernel of the regular spread is a group of order $q+1$ that leaves each line of the regular spread invariant, and cyclically permutes the points of each line. So if $\phi$ is a generator of the Bruck kernel, the orbit of any line of $PG(3,q)$ not in $S$ under $\phi$ is an opposite regulus for some regulus contained in $S$. Taking orbits under $\psi = \phi^2$, we cyclically obtain opposite half-reguli. If a $t$-nest is hemi-replaceable such that the opposite half-reguli are orbits under $\psi$, we say the nest is {\em Bruck-replaceable}. All of the nests in $PG(3,5)$ from known families are Bruck-replaceable.

From the construction of the $2(q-1)$-nests of Baker and Ebert, we have seen an instance where the union of two pairwise disjoint nests is still a nest. Trivially, if two disjoint nests are replaceable, then their union is also a replaceable nest. If two nests are not disjoint, it is still possible for their union to be replaceable, but we need to be more careful in how we proceed. We begin with a straightforward lemma describing the replacement sets for hemi-replaceable nests.

\begin{lemma}
\label{halflemma}
Let $N$ be a hemi-replaceable $t$-nest in the regular spread $S$ of $PG(3,q)$. Let $R_1 \ldots R_t$ be the reguli in $N$, and $R'_1 \ldots R'_t$ their corresponding opposite reguli. If $H=\{S_1 \ldots S_t\}$ is a hemi-replacement for $N$, then $H'=\{S'_1 \ldots S'_t\}$, where $S'_i = R'_i \setminus S_i$, is also a hemi-replacement for $N$.
\end{lemma}

\begin{proof}
It is clear that $S'_i$ is a subset of $R'_i$ and has size $\frac12 (q+1)$ for all $i$; it remains to show that all lines in $N'=\bigcup_{i=1}^t S'_i$ are pairwise disjoint. Let $x,y \in N'$ be distinct lines, such that $x \in S'_i$ and $y \in S'_j$. Clearly $x$ and $y$ are disjoint if $i=j$, so we may assume $i \ne j$. Also if $R_i$ and $R_j$ are disjoint, $x$ and $y$ must be disjoint, so we may assume there exists a line $\ell$ contained in both $R_i$ and $R_j$, and that if $x$ and $y$ meet, it must be in a point of $\ell$. Each line of $S_i$ must meet $\ell$ in a point, so the set $S_i^* = \{\ell \cap m: m \in S_i\}$ is a set of $\frac12 (q+1)$ points of $\ell$. Similarly, $S_j^* = \{\ell \cap n: n \in S_j\}$ is a set of $\frac12 (q+1)$ points of $\ell$. $S_i^*$ and $S_j^*$ must be disjoint, for otherwise there would be lines $m \in S_i \in H$ and $n \in S_j \in H$ of the hemi-replacement that intersect. Therefore we can write $\ell = S_i^* \cup S_j^*$. The line $x$ is in $S'_i$, and since all lines in $R'_i$ are disjoint, $x$ cannot contain a point of $S_i^*$, thus $x$ must meet $\ell$ in a point of $S_j^*$. An analogous argument shows that $y$ must meet $\ell$ in a point of $S_i^*$, showing that $x$ and $y$ cannot meet on $\ell$. Therefore, $x$ and $y$ must be disjoint, proving the lemma.
\end{proof}

With this lemma in place, we can describe our construction.

\begin{thm}
Let $N_1$ and $N_2$ be Bruck-replaceable nests of reguli in the regular spread $S$ of $PG(3,q)$, such that $N_1$ and $N_2$ have no lines in common except those on exactly one regulus $R$. Then the set $N = N_1 \cup N_2 \setminus\{R\}$ is a Bruck-replaceable nest.
\end{thm}

\begin{proof}
Let $\ell$ be any line covered by $N$, and assume without loss of generality that $\ell$ is covered by $N_1$; the analogous argument will hold if $\ell$ is covered by $N_2$. If $\ell$ is also covered by $R$, then there exists exactly one regulus in $N \cap N_1$ which covers $\ell$, but there also exists exactly one regulus in $N \cap N_2$ that covers $\ell$. These reguli must be distinct since $N \cap N_1$ and $N \cap N_2$ are disjoint, so in this case exactly two reguli in $N$ cover $\ell$. If $\ell$ is not contained in $R$, then there are exactly two reguli in $N \cap N_1$ that cover $\ell$. Since $N_1$ and $N_2$ meet only in $R$, $\ell$ cannot lie in any regulus of $N_2$, so again $\ell$ is contained in exactly two reguli of $N$. Hence $N$ is a nest.

Let $H_1$ be a Bruck hemi-replacement of $N_1$, and let $R'$ be the opposite half-regulus to $R$ contained in $H_1$. Since $H_1$ is a Bruck hemi-replacement, $R'$ must be an orbit under the mapping $\psi$ in the Bruck kernel. Let $H_2$ be a Bruck hemi-replacement of $N_2$. Without loss of generality, we may assume that $H_2$ contains $R''$, the opposite half-regulus to $R$ that is the orbit under $\psi$ distinct from $R'$; if this were not the case, then $H_2$ would have to contain $R'$, and we could apply Lemma~\ref{halflemma} to create a Bruck hemi-replacement for $N_2$ containing $R''$. We claim $H = H_1 \cup H_2 \setminus \{R',R''\}$ is a Bruck hemi-replacement for $N$.

Most of the claim is obvious: the elements of $H$ are all orbits under $\psi$, and there are the right number of them. It only remains to show that the lines in $H$ are pairwise disjoint. First note that if $x$ and $y$ are distinct lines of $H$ covered by reguli in $H_1$ (resp. $H_2$), then $x$ and $y$ are disjoint since $H_1$ (resp. $H_2$) is a Bruck hemi-replacement for $N_1$ (resp. $N_2$). Hence if $x$ and $y$ are distinct lines of $H$, we need only consider the case where $x \in R'_1$ and $y \in R'_2$ for some opposite half-reguli $R'_1 \in H_1$ and $R'_2 \in H_2$, where $R_1$ is a regulus of $N_1$ and $R_2$ is a regulus of $N_2$. As in the proof of Lemma~\ref{halflemma}, the only way $x$ and $y$ can meet is if there exists a line $m \in R_1 \cap R_2$ on which their intersection must lie. Since $R_1 \in N_1$ and $R_2 \in N_2$, this line $m$ must be contained in the regulus $R$.

Let $M'$ (resp. $M''$) be the set of points on $m$ contained in the lines of $R'$ ($R''$), the opposite half-regulus to $R$ contained in $H_1$ (resp. $H_2$). $M'$ and $M''$ are clearly disjoint. Since $H_1$ is a hemi-replacement that contains both $R'$ and $R'_1$, we see that the line $x$ must meet $m$ in a point of $M''$. The analogous argument shows that $y$ must meet $m$ in a point of $M'$, hence $x$ and $y$ must be disjoint. This proves the claim.
\end{proof}

Since we have a list of all nests in the regular spread of $PG(3,5)$, it is not difficult to determine which can be obtained from this construction. In this case, there are three examples. Two of these are Baker-Ebert-Weida single tabs, and they generate the planes $A_2$ and $B_7$ upon replacement, as noted previously. The other is the union of a Bruen chain and a Baker-Ebert mixed 5-nest (which when replaced on its own yields $A_6$), and the resulting nest generates $B_3$.

As an example of the robustness of this technique, we executed an exhaustive search for all nests in the regular spread of $PG(3,7)$, which generated 85 examples. Of these 85 examples, 59 are Bruck-replaceable and 12 can be obtained as the union of two smaller Bruck-replaceable nests meeting in a regulus. Of the twelve planes of order 49 generated via nest placement, only two are of a type discussed in the previous section: a Pabst-Sherk plane, and the Rao, Rodabaugh, Wilke and Zemmer plane.

\section{Looking for $B_8$}
With the constructions given in the previous section, we have placed all of the translation planes of order 25, except for $B_8$, into our family tree. In the previous section, we examined the general construction of nest replacement beyond the known infinite families of nests with success, so perhaps applying the same strategy to other more general construction techniques may be fruitful. Perusing Johnson, et al.'s~\cite{jjb} {\em Handbook of Finite Translation Planes}, we were able to identify the general techniques of flocks of quadratic cones, hyperbolic fibrations, $j$-planes and multiple nests as potential candidates. However, we already know that there is only one non-linear conical flock, and thus one non-linear hyperbolic fibration, in $PG(3,5)$, and we have already identified those planes. The $j$-planes are known to be spawned from a hyperbolic fibration~\cite{jpw}, and thus coincide with the Pabst-Sherk/$(q+1)$-nest planes when $q=5$. 

The remaining technique of multiple nests, first explored by Jha and Johnson~\cite{jj}, provides a generalization of nests by defining a $(k;t)$-nest to be a set of $t$ reguli such that each line contained in a regulus of the multiple nest lies in exactly $2k$ such reguli. The definition of multiple nests was motivated to support a classification of translation planes of order $q^2$ admitting a linear group of order $q(q+1)$, and in this case the only planes obtained from multiple nests are conical flock planes, so $B_8$ is not a multiple nest plane either. However the concept of ``higher-order" nests seems promising.

Abstracting the concepts of subregular regulus replacement, nest replacement, and multiple nest replacement, we define a $k$-web ${\mathcal W}$ to be a set of reguli in a regular spread such that every line contained in one of the reguli of ${\mathcal W}$ lies in exactly $k$ distinct reguli of ${\mathcal W}$. In this formulation, we see that a set of pairwise disjoint reguli in a regular spead is exactly a $1$-web, while a nest in a regular spread is a $2$-web, and a $(k;t)$-multiple nest is a $2k$-web. Deriving a spread obtained from nest replacement is a hybrid, where we have a disjoint $1$-web and $2$-web being replaced. Without any other extension, 17 out of the 20 non-regular spreads in $PG(3,5)$ can be obtained by these replacements, a surprisingly robust number.

Given a $k$-web ${\mathcal W}$ of a regular spread $S$ of $PG(3,q)$, let $t$ denote the number of reguli in ${\mathcal W}$ and let $\phi$ be a generator of the Bruck kernel of $S$. For any regulus $R$ of $S$, we call any of the line orbits in the opposite regulus of $R$ under $\phi^{\alpha}$ an {\em $\alpha$-semitransversal} of $R$.  We define a {\em Bruck replacement} for ${\mathcal W}$ to be a set of $t$ $k$-semitransversals, one for each regulus in ${\mathcal W}$, such that the set of points covered by the lines contained in the reguli of ${\mathcal W}$ is exactly the set of points covered by the lines contained in the $k$-semitraversals of the Bruck replacement.

When $k=1$, a $k$-web is a set of $t$ pairwise disjoint reguli in $S$, and a Bruck replacement is a set of $t$ $1$-semitransversals, which are exactly the opposite reguli to the reguli in the $k$-web. When $k=2$, a $k$-web is a nest, and a Bruck replacement is exactly a Bruck replacement for the nest as defined previously. 

Clearly, Bruck replacement for $k$-webs only makes sense when $k$ divides $q+1$, so for the case $q=5$ that we are examining, $k=3$ is a valid choice. Using code very similar to that used to identify nests in $PG(3,5)$, we executed a search for $3$-webs in a regular spread of $PG(3,5)$, with shocking results. Our search identified 25 $3$-webs in the regular spread of $PG(3,5)$, up to isomorphism. Equally surprising was the determination that 15 of these $3$-webs were Bruck replaceable, yielding 13 non-isomorphic spreads of $PG(3,5)$ including the spread which generates $B_8$. Interestingly, two non-isomorphic $3$-webs can be Bruck replaced to obtain $B_8$, with the Hall plane being the only other plane sharing this property. It is also possible to get another regular spread by replacing a $3$-web.

Table~\ref{T2} provides a description of which translation planes of order 25 can be obtained by web replacement.
\noindent
\begin{table}
\centering
\caption{Web replacement in the regular spread of $PG(3,5)$}{\label{T2}}
\begin{tabular}{|c|c|c|c|c|}
\hline
Label & $1$-web & $2$-web & $1$-web + $2$-web & $3$-web\\ \hline\hline
$S_1$ &  &  & & x\\ \hline
$S_2$ & x &  & & x\\ \hline
$S_3$ & x & & & x\\ \hline
$S_4$ & x & & & x\\ \hline
$S_5$ & x & x & &\\ \hline\hline
$A_1$ & & x & & x\\ \hline
$A_2$ & & x & &\\ \hline
$A_3$ & & x & &\\ \hline
$A_4$ & & x & & x\\ \hline
$A_5$ & & & x & x\\ \hline
$A_6$ & & x & &\\ \hline
$A_7$ & & & x &\\ \hline
$A_8$ & & & x & x\\ \hline\hline
$B_1$ & & & & x\\ \hline
$B_2$ & & & &\\ \hline
$B_3$ & & x & &\\ \hline
$B_4$ & & x & & x\\ \hline
$B_5$ & & x & &\\ \hline
$B_6$ & & x & & x\\ \hline
$B_7$ & & x & & x\\ \hline
$B_8$ & & & & x\\ \hline
\end{tabular}
\end{table}

\section{Conclusion}
Given the number and breadth of construction techniques for finite translation planes, we began this research with the assumption that all of the translation planes of order 25 belonged to at least, and probably more than, one infinite family; we were genuinely dumbfounded to find this not to be the case, and that in fact two of these planes are not part of known infinite families. While we have tried to provide some additional context around these planes, the results presented here really provide directions for additional research, rather than representing a culmination of the work.

A next step in this program would be to tackle the translation planes of 49, of which there are 1347, based on a search conducted by Mathon and Royle~\cite{mr}. If the proportion of translation planes of order 49 not in a described infinite family follow the pattern for the order 25 case, we should expect 192 of these planes to not be in known families, but the real number is likely much higher. We believe this represents a great opportunity to discover and extrapolate new construction techniques for finite translation planes.

\bibliographystyle{plain}

\end{document}